\documentclass[reqno,12pt]{article}
\usepackage{color}
\usepackage{amssymb,amsmath,amsthm,amsfonts}
\usepackage{fourier}
\usepackage{indentfirst}
\usepackage{hyperref}
\usepackage{xy}
\usepackage{graphicx}
\usepackage{float}
\usepackage{psfrag}
\usepackage{url}
\usepackage{physics}
\usepackage{mathrsfs}
\usepackage{tabularx}
\usepackage{booktabs}
\usepackage[labelfont=bf,format=plain,justification=centering,singlelinecheck=false]{caption}
\usepackage{hyperref}
\hypersetup{
	colorlinks=true,
    breaklinks=true,
	linkcolor=blue,    
	urlcolor=blue,
	citecolor=cyan,
}
\usepackage[
backend=biber,
giveninits=true,
style=alphabetic,
maxnames=5,
url=false,
doi=false,
isbn=false
]{biblatex}

\newbibmacro{string+doiurlisbn}[1]{%
	\iffieldundef{doi}{%
		\iffieldundef{url}{%
			\iffieldundef{isbn}{%
				\iffieldundef{issn}{%
					#1%
				}{%
					\href{https://books.google.com/books?vid=ISSN\thefield{issn}}{#1}%
				}%
			}{%
				\href{https://books.google.com/books?vid=ISBN\thefield{isbn}}{#1}%
			}%
		}{%
			\href{\thefield{url}}{#1}%
		}%
	}{%
		\href{https://doi.org/\thefield{doi}}{#1}%
	}%
}

\DeclareFieldFormat{title}{\usebibmacro{string+doiurlisbn}{\mkbibemph{#1}}}
\DeclareFieldFormat[article,incollection]{title}%
{\usebibmacro{string+doiurlisbn}{\mkbibquote{#1}}}

\renewbibmacro{in:}{}

\DeclareFieldFormat[article]{volume}{\mkbibbold{#1}}
\DeclareFieldFormat[article]{number}{\mkbibparens{#1}}
\DeclareFieldFormat{pages}{#1}	

\renewbibmacro{journal+issuetitle}{%
	\usebibmacro{journal}%
	\setunit{\addspace}%
	\setunit{\addcomma\addspace}
	\printtext{\printdateextra}
	\setunit{\addcomma\addspace}%
	\printfield{volume}%
	\printfield{number}%
	\newunit
}

\renewbibmacro{note+pages}{%
	\iffieldundef{volume}{\iffieldundef{year}{}{\addcomma\addspace}}{\addcolon\addspace}
	\printfield{pages}%
	\iffieldundef{note}{}{\addcomma\addspace \printfield{note}}
	\setunit*{\addcomma\space}%
	\newunit
}

\usepackage[a4paper,text={160true mm,230true mm},centering,top=35true mm,head=5true mm,headsep=2.5true mm,foot=8.5true mm]{geometry}
\usepackage{lipsum}

\makeatletter \@addtoreset{equation}{section} \makeatother

\renewcommand\thetable{\thesection.\@arabic\c@table}

\theoremstyle{plain}
\newtheorem{maintheorem}{Theorem}

\newtheorem{theorem}{Theorem}[section]
\newtheorem{proposition}{Proposition}[section]
\newtheorem{lemma}{Lemma}[section]

\newtheorem{remark}{Remark}[section]

\newtheorem*{theorem*}{Main Theorem}
\newtheorem*{claim}{Claim}

\newcommand{\N}{\mathbb{N}}
\newcommand{\R}{\mathbb{R}}

\newcommand{\diam}{\operatorname{diam}}

\newcommand{\graph}{\operatorname{graph}}
\newcommand{\Lip}{\operatorname{Lip}}
\newcommand{\loc}{\operatorname{loc}}

\newcommand{\cI}{{\mathcal I}}

\newcommand{\cP}{\mathcal{P}}

\newcommand{\cQ}{\mathcal{Q}}

\newcommand{\cT}{\mathcal{T}}

\newcommand{\cS}{\mathcal{S}}

\allowdisplaybreaks[3]
\begin{document}
\begin{sloppypar}
	\title{Upper semi-continuity of metric entropy for diffeomorphisms with dominated splitting}
	
\author{Chiyi Luo, Wenhui Ma and Yun Zhao\footnote{
	This work was partially supported by National Key R\&D Program of China (2022YFA1005801).
Y. Zhao is partially supported by NSCF(12271386). }}
	\date{}    
	
	\maketitle
		
	\begin{abstract}
		For a $C^{r}$ $(r>1)$  diffeomorphism on a compact manifold that admits a dominated splitting, this paper establishes the upper semi-continuity of the entropy map. More precisely, this paper establishes the upper semi-continuity of the entropy map in the following two cases:
		(1) if a sequence of invariant measures has only positive Lyapunov exponents along a sub-bundle and non-positive Lyapunov exponents along another sub-bundle,  then the upper limit of their metric entropies is less than or equal to the entropy of the limiting measure; (2) if an invariant measure has positive Lyapunov exponents along a sub-bundle and non-positive Lyapunov exponents along another sub-bundle, then the entropy map is upper semi-continuous at this measure.
	\end{abstract}
	   
   \tableofcontents

   \section{Introduction}\label{SEC:1}
   Entropy is an important invariant quantity in dynamical systems, which plays a central role in understanding the chaotic behavior of a system.  The metric entropy  measures the average amount of information and complexity in the system, and the topological entropy characterizes the exponential growth rate of the number of periodic points.   Given a dynamical system $(M,f)$, the  map $\mu\mapsto h_{\mu}(f)$ defined on  the set of all $f$-invariant measures is called entropy map in this paper.   In the entropy theory of dynamical systems,  it is an interesting topic to study the continuity of entropy map.  However,  the entropy map is generally not lower semi-continuous, even in uniformly hyperbolic systems. 
   For example, by shadowing lemma, one can construct a sequence of ergodic measures supported on periodic orbits, but it converges to an invariant measure with positive entropy.  This paper studies the upper semi-continuity of entropy map for a class of non-uniformly hyperbolic dynamical systems.

   For $C^{\infty}$ diffeomorphisms on a compact manifold, Newhouse \cite{Newhouse1989} proved that the entropy map is upper semi-continuous.  For $C^{r}$ diffeomorphisms on a compact manifold with finite positive $r$, the upper semi-continuity of the entropy map may fail, e.g., 
   Misiurewicz \cite{Misiurewicz1973} gave examples of diffeomorphisms on a four dimensional compact manifold, and  examples in the two dimensional case are given by Buzzi \cite{Buzzi2014}. However, one can establish the upper semi-continuity of the entropy map for smooth maps  under certain conditions. 
   For $C^r$ surface diffeomorphisms, Burguet \cite{Burguet2024} proved that the entropy map is upper semi-continuous at ergodic measures with large entropy.
   For $C^1$ diffeomorphisms away from homoclinic tangencies,  in \cite{Liao2013} the authors proved that the entropy map is upper semi-continuous. Under some additional conditions, in \cite{Liao2015} the authors showed that the entropy map is upper semi-continuous for diffeomorphisms admit a dominated splitting. 
  
   In this paper,   for a $C^r$ diffeomorphism with dominated splitting, we will show the upper semi-continuity of the entropy map provided that the Lyapunov exponents of the invariant measures statisfy some conditions which are weaker than that of \cite{Liao2015}, see Remark \ref{LS} for detailed descriptions. 
     
  Now we will give some necessary notions to present the statements of the main results in this paper.
   Let $f:M\rightarrow M$ be a $C^r$ $(r>1)$ diffeomorphism on a compact Riemannian manifold $M$ without boundary.
   We say that an $f$-invariant compact subset $\Lambda\subset M$  admits a dominated splitting of $f$,  if there exist $\lambda>0$ and a continuous $Df$-invariant splitting $T_{\Lambda}M=E(x)\bigoplus F(x)$ such that 
   $$\dfrac{m(D_xf|_{F(x)})}{ \|D_xf|_{E(x)}\|}\geq e^{\lambda}, \forall x\in \Lambda$$
   where $\|A\|$ is the norm and $m(A):=\inf_{\|v\|=1} \|A(v)\|$ is the co-norm of a linear map $A$.
   
   For every $f$-invariant measure $\mu$ supported on $\Lambda$, by the sub-additive ergodic theorem, the following limits 
   $$\lambda_F^{\min}(x):=\lim_{n\rightarrow  \infty} \frac{1}{n} \log m(D_xf^n|_{F(x)}),~\lambda_E^{\max}(x):=\lim_{n\rightarrow  \infty} \frac{1}{n} \log \|D_xf^n|_{E(x)}\|$$
   are well-defined for $\mu$-almost every $x\in \Lambda$. Moreover, if $\mu$ is ergodic then $\lambda_F^{\min}(x)$ and  $\lambda_E^{\max}(x)$ are constants for $\mu$-almost every $x\in \Lambda$, which are denoted by 
$\lambda_F^{\min}(\mu)$ and $\lambda_E^{\max}(\mu)$ respectively.
   
   \begin{maintheorem}\label{Thm:A}
   	Let $f:M\rightarrow M$ be a $C^r\, (r>1)$ diffeomorphism  on a compact Riemannian  manifold, and let $\Lambda\subset M$ be an $f$-invariant compact subset that admits a dominated splitting $T_{\Lambda}M=E \bigoplus F$. 
   	Then, there exists $\varepsilon_f>0$ such that if $\nu$ is an ergodic measure supported on $\Lambda$ satisfies that 
   	\begin{enumerate}
   		\item[$ \bullet$]   $\lambda_E^{\max}(\nu)\leq 0$ and $\lambda_F^{\min}(\nu)>0$,
   		\item[$ \bullet$] $\cQ$ is a finite partition with $\nu(\partial \cQ)=0$ and $\diam(\cQ)<\varepsilon_{f}$, 
   	\end{enumerate}
   	we have that $h_{\nu}(f)=h_{\nu}(f,\cQ)$.
   \end{maintheorem}
   
   As a consequence, one can show the following theorem.
   
   \begin{maintheorem}\label{Thm:B}
   Let $f:M\rightarrow M$ be a $C^r\, (r>1)$ diffeomorphism  on a compact Riemannian  manifold, and let $\Lambda\subset M$ be an $f$-invariant compact subset that admits a dominated splitting $T_{\Lambda}M=E \bigoplus F$.  If  a sequence of invariant measures $(\nu_k)_{k=1}^{\infty}$  satisfies 
   	\begin{enumerate}
   		\item[$ \bullet$] for each $k>0$, one has $\nu_k(\Lambda)=1$, $\lambda_E^{\max}(x)\leq 0$ and $\lambda_F^{\min}(x)>0$ for $\nu_k$-almost every $x$;
   		\item[$ \bullet$] $\nu_k$ converges to $\mu$ in the  weak star topology,
   	\end{enumerate}
   	then we have that 
   	$\displaystyle{\limsup\limits_{k\to\infty} h_{\nu_k}(f)\leq h_{\mu}(f)}$.   	
   \end{maintheorem}

\begin{remark}  {\rm \label{LS} Under the set up of the above theorem,  given a sequence of invariant measures $\{\nu_k\}$, if there exists $\chi>0$ such that: (1)$\lambda_E^{\max}(x)< -\chi$ and $\lambda_F^{\min}(x)>\chi$ for $\nu_k$-almost every $x$ for every $k>0$; (2) $\nu_k$ converges to $\mu$ in the  weak star topology, and $\lambda_E^{\max}(x)< -\chi$ and $\lambda_F^{\min}(x)>\chi$ for $\mu$-almost every $x$, then the same conclusion holds. See \cite{Liao2015} for the details.  We point out that they only require that $f$ is a $C^1$ diffeomorphism.
   Our results do not require $\lambda_E^{\max}(x), \lambda_F^{\min}(x)$ to be uniformly far from zero, and there are no conditions imposed on the limiting measure.}
   \end{remark}

   Finally, we show the upper semi-continuity of the entropy map under some conditions on the limiting measure. 
   
   \begin{maintheorem}\label{Thm:C}
   	 Let $f:M\rightarrow M$ be a $C^r\, (r>1)$ diffeomorphism  on a compact Riemannian  manifold, and let $\Lambda\subset M$ be an $f$-invariant compact subset that admits a dominated splitting $T_{\Lambda}M=E \bigoplus F$.
   	Assume that $\mu$ is an $f$-invariant measure supported on $\Lambda$ with
   	$\lambda_E^{\max}(x)\leq 0$ and $\lambda_F^{\min}(x)>0$ for $\mu$-almost every $x$.
   	Then, the entropy map  is upper semi-continuous at $\mu$.
   \end{maintheorem}
   
   In particular, for partially hyperbolic systems, if the center bundle decomposes into a dominated way of one-dimensional subbundles, then the entropy map is upper-semi continuous. 
   This result can also be deduced from \cite{Buzzi2018}.

   Theorem \ref{Thm:C} is not new, it only requires  that $f$ is $C^1$, as shown in \cite{Liao2020}, for instance. 
   In this paper, we utilize the results of Ledrappier and Young \cite{Young85-2} (valid for $C^2$ diffeomorphisms) and those in \cite{Brown2022} which extend to $C^{1+}$ diffeomorphisms.
   Therefore, we assume that $f$ is $C^r$ with $r>1$.
      
   Tail entropy is used to estimate the degree to which entropy fails to be upper semi-continuous.
   It quantifies  how many arbitrarily small Bowen balls are required to cover a  Bowen ball of fixed-size.
   By Ledrappier and Young's result \cite{Young85-2}, entropy only arises along unstable manifolds. 
   Then, we estimate the “tail entropy” along unstable manifolds, which exhibit a certain degree of expansiveness.
   
   In Burguet \cite{Burguet2024}, Yomdin’s theory is used to provide a uniform diameter. 
   However, in our approach, we apply dominated splitting to establish a uniform diameter instead. For any local unstable manifold with this uniform size,  there will be no bending along the strong subbundle direction after iterations of the dynamics.

  The rest of the paper is organized as follows. In Section \ref{Sec:entropy}, we review basic properties of partial entropy along unstable manifolds. Section \ref{Sec:Dominated} establishes a uniform diameter bound for partitions, as stated in Theorem \ref{Thm:A}, using the properties of dominated splitting. 
  Section \ref{Sec:PofBC} provides the proof of Theorems \ref{Thm:B} and  \ref{Thm:C}, by assuming the validity of Theorem \ref{Prop:Key} which is  a partial entropy version of Theorem \ref{Thm:A}. Section \ref{Sec:PofT4} and Section \ref{Sec:PofP5} are devoted to the proof of Theorem \ref{Prop:Key}.
   
    \section{Preliminaries}\label{Sec:entropy} 
    In this section, we will recall some basic properties of entropy and some useful results that are used in the proof of the main results.

   Let $(M,f)$ be given as in the previous section.  For a probability measure $\mu$ (not necessarily invariant) and a finite partition $\mathcal P$, define the \emph{static entropy} of $\mu$ as
    $$H_\mu(\mathcal P)=\sum_{P\in\mathcal P}-\mu(P)\log\mu(P)=\int -\log\mu(P(x)){\rm d}\mu(x).$$
    For  $n\ge 1$, let
    $${\mathcal P}^n:={\mathcal P}^{n,f}=\bigvee_{j=0}^{n-1}f^{-j}(\mathcal P).$$
    
    \begin{lemma}[\cite{Walters82} Section 8.2] \label{Lem:11}
    	Let $\mu$ be a probability measure and $\mathcal P$  a finite partition.  
    	Then, for every $n>0$ and every $0<m<n$, we have
    	\begin{equation}
    		\dfrac{1}{n}H_{\mu}(\cP^n)\leq \dfrac{1}{m}H_{\mu_n}(\cP^m)+\dfrac{2m}{n} \log \# \cP
    	\end{equation}
    	where $\displaystyle{\mu_n:=\frac{1}{n} \sum_{k=0}^{n-1} \mu \circ f^{-k}}$ and $\# \cP$ is the cardinality of $\cP$.
    \end{lemma}  
        
    For an $f$-invariant measure $\mu$ and a finite partition $\mathcal P$, the metric entropy of $\mu$ with respect to $\mathcal P$ is defined as follows: 
    $$h_\mu(f,\mathcal P)=\lim_{n\to\infty}\frac{1}{n}H_\mu({\mathcal P}^n);$$
    and the \emph{metric entropy} of $\mu$ is defined as 
    $$h_\mu(f)=\sup \Big\{h_\mu(f,\mathcal P): \cP~\text{is a finite partition} \Big\}.$$
    
    Recall the Oseledec theorem \cite{Oseledec68}: Let $\mu$ be an $f$-invariant ergodic measure, there are finite numbers 
    $$\lambda_1(\mu,f)>\lambda_2(\mu,f)>\cdots>\lambda_t(\mu,f)$$
    and a $Df$-invariant splitting over a full $\mu$-measure set 
    $$E^1\bigoplus E^2\bigoplus\cdots\bigoplus E^t$$
    such that $\displaystyle{\sum_{i=1}^t \dim E^i=\dim M}$, 
    and for each non-zero vector $v\in E^i$
    $$\lim_{n\to\pm\infty}\frac{1}{n}\log\|D_{x}f^n v\|=\lambda_i(\mu,f),~\mu\text{-a.e.}\ x.$$
    When there is no confusion caused, one denotes $\lambda_i(\mu,f)$ by $\lambda_i(\mu)$ simply.
    
    For each $1\leq i \leq t$ with $\lambda_i(\mu)>0$, let  
    \begin{equation}\label{eq:SB}
    	E^{u,i}(x)=\bigoplus_{j=1}^{i}E^j(x),~E^{s,i}(x)=\bigoplus_{j=i+1}^{t}E^j(x).
    \end{equation}
    By the theory of unstable manifolds \cite{Pesin07}, for $\mu$-almost every $x$,
    \begin{equation}\label{eq:Unstable}
        W^{u,i}(x):=\left\{y\in M: \limsup_{n\to\infty} \frac{1}{n} \log d(f^{-n}(x),f^{-n}(y))\leq -\lambda_i(\mu) \right\}
    \end{equation}
    is a $C^{r}$ immersed submanifold tangent to $E^{u,i}(x)$ and inherits a Riemannian metric from $M$.
    We denote this distance by $d_x^{u,i}$. 
    With this distance we define $(n,\rho)$-Bowen balls by 
    $$ V^{u,i}(x,n,\rho):=\left\{y\in W^{u,i}(x): d_{f^j(x)}^{u,i}(f^j(x),f^j(y))<\rho,~\forall 0\leq  j <n \right\}.$$
        
   By \cite{Ledrappier82} and \cite{Young85-2}, there exists a measurable partition $\xi$ subordinate to $W^{u,i}$ with respect to $\mu$, meaning that  $\xi(x)\subset W^{u,i}(x)$ and $\xi(x)$ contains an open neighborhood of $x$ in $W^{u,i}(x)$ for $\mu$-almost every point $x$, where $\xi(x)$ denotes the element of $\xi$ which contains $x$.
    
    From Rokhlin \cite{Rohlin67}, for the measurable partition $\xi$, there is a family of conditional measures $\{ \mu_{\xi(x)}\}$ such that: (1) $\mu_{\xi(x)}$ is supported on $\xi(x)$; 
    (2) $\displaystyle{\mu(A)=\int \mu_{\xi(x)} (A) {\rm d}\mu }$ for every measurable subset $A$.
    
    As in Ledrappier and Young \cite{Young85-2}, the partial entropy along $W^{u,i}$ is defined as follows: 
    \begin{equation} \label{eq:LY}
    	h_{\mu}^{i}(f)=\lim_{\rho \rightarrow 0} \liminf_{n\rightarrow \infty} -\frac{1}{n} \log \mu_{\xi(x)}(V^{u,i}(x,n,\rho))=\lim_{\rho \rightarrow 0}\limsup_{n\rightarrow \infty} -\frac{1}{n} \log \mu_{\xi(x)}(V^{u,i}(x,n,\rho)).
    \end{equation}
    This limit exists and is constant $\mu$-almost everywhere. 
    For each $1\leq i\leq t$,
    \begin{enumerate}
    	\item[(1)]  if $\lambda^{i}(\mu)>0$ and $\lambda^{i+1}(\mu)\leq 0$, then $h_{\mu}^{i}(f)=h_{\mu}(f)$;
    	\item[(2)]  if $\lambda^{i}(\mu)>\lambda^{i+1}(\mu)>0$, then $h_{\mu}^{i+1}(f)\leq h_{\mu}^{i}(f)+\dim E^{i+1}\cdot \lambda^{i+1}(\mu)$.
    \end{enumerate}    
    The usual Bowen ball is defined by 
    $$B(x,n,\rho):=\Big\{y\in M: d(f^j(x),f^j(y))<\rho, \forall 0\leq j< n\Big\}.$$
    Note that  $\mu_{\xi(x)}(V^{u,i}(x,n,\rho))\leq \mu_{\xi(x)}(B(x,n,\rho))$ for $\mu$-almost every $x$. Now we will recall two results in \cite{Yang2024}, which are useful in the proof of the main results in this paper.
    
    \begin{proposition}[\cite{Yang2024}, Proposition 2.1]\label{Prop:Two Balls}
    	For every $\varepsilon>0$ and every $\delta>0$, there exist a subset $K\subset M$ with $\mu(K)>1-\delta$ and $\rho>0$, such that 
    	$$ h_{\mu}^{i}(f)\leq \liminf_{n\rightarrow +\infty} -\frac{1}{n}\log \mu_{\xi(x)}\left(K\cap B(x,n,\rho)\right)+\varepsilon, \quad \forall x\in K.$$
    \end{proposition}
    By Proposition \ref{Prop:Two Balls}, for $\varepsilon>0$ one can choose a compact set $K$ with $\mu(K)>1/2$ and $\rho>0$ such that 
    \begin{equation}\label{eq:LimK}
    	h_{\mu}^{i}(f)\leq \liminf_{n\rightarrow +\infty} -\frac{1}{n}\log \mu_{\xi(x)}\left(K\cap B(x,n,\rho)\right)+\varepsilon,\quad \forall x\in K.
    \end{equation}
    Choose $x_0\in K$ and a measurable set $\Sigma \subset W^{u,i}_{\loc}(x_0)$ with $\mu_{\xi(x_0)}(K\cap \Sigma)>0$,  let $\omega(\cdot):=\mu_{\xi(x_0)}(\cdot \cap K\cap \Sigma)/\mu_{\xi(x_0)}( K\cap \Sigma)$, Luo and Yang showed the following result in \cite[Proposition 2.2]{Yang2024}.
    \begin{proposition}\label{Pro:unstable-entropy-partition}
    	Let $\varepsilon$, $\rho$, $K$, $\Sigma$ and $\omega$ be chosen as described above,
    	for any finite partition $\cal P$ satisfying ${\rm Diam}({\cal P})<\rho$ one has
    	$$h_{\mu}^{i}(f)\le\liminf_{n\to\infty}\frac{1}{n}H_{\omega}(\cP^n)+\varepsilon.$$   	
    \end{proposition}
      
    \section{Auxiliary results}\label{Sec:Dominated}
    In this section, we consider an $f$-invariant compact subset $\Lambda$ admits a dominated splitting $T_{\Lambda}M=E\bigoplus F$, i.e.,  there exists $\lambda>0$ such that 
    $$\dfrac{m(D_xf|_{F(x)})}{ \|D_xf|_{E(x)}\|}> e^{\lambda},\, \forall x\in \Lambda.$$
    One can extend the dominated splitting $E\bigoplus F$ to an open neighborhood $U$ of $\Lambda$ such that 
    $$\dfrac{m(D_xf|_{F(x)})}{ \|D_xf|_{E(x)}\|}\geq e^{\lambda},\, \forall x\in \overline{U}.$$
    Denote by $\pi^{E}_x$ and $\pi^{F}_x$ the projections from $T_xM=E(x) \bigoplus F(x)$ to $E(x)$ and $F(x)$ respectively.
    By the continuity of the splitting, there exists a constant $C_{\rm ang}\geq 1$ such that 
    \begin{equation}\label{eq:Projec}
    	\max \{\|\pi^{E}_x\|, \|\pi^{F}_x\| \}\leq  C_{\rm ang},\, \forall x\in \overline{U}.
    \end{equation}
    
   For every $x$ and every $\rho>0$, let
   $$B_x(\rho):=\{v\in T_xM: \|v\|<\rho \},~R_x(\rho):=\{v\in T_xM: \max\{\|\pi^{E}_x(v)\|, \|\pi^{F}_x(v)\|\}<\rho \}.$$  
   Since $M$ is  compact, one can choose $\rho(M)>0$ and $r(M)>0$ such that for every $x\in M$
   \begin{enumerate}
   	\item[$ \bullet$]  $\exp_x: B_x(\rho(M)) \rightarrow M$ is a $C^{\infty}$ embedding;
   	\item[$ \bullet$]  $\exp_x(B_x(\rho(M)))\supset B(x,r(M))$;
   	\item[$ \bullet$]  $\|D_v(\exp_x)\|\leq 2,\, \forall v\in B_x(\rho(M))$ and $\|D_y(\exp_x^{-1})\|\leq 2,\, \forall y\in B(x,r(M))$.
   \end{enumerate}     
   We denote $F_x=\exp_{f(x)}^{-1}\circ f \circ \exp_x$, $F_x^i=F_{f^{i-1}x}\circ\cdots\circ F_{fx}\circ F_x$ and $F_x^{-i}:=(F^i_{f^{-i}(x)})^{-1}$ for each $i\ge 1$.         
   Since $f$ is $C^r$ smooth, the following properties hold:
   \begin{enumerate}
   	\item[$ \bullet$]  $F_x(0)=0$ and $D_0(F_x)=D_xf$;
   	\item[$ \bullet$]  for every $\eta>0$, there is a $\hat \varepsilon>0$ such that the map $F_x: R_x(\hat \varepsilon)\to M$ satisfies that 
   	$$ \Lip(F_x-D_xf)\leq \eta ~ \text{and} ~ \sup_{v\in R_x(\hat \varepsilon)}\|D_v(F_x)-D_0(F_x)\|\leq \eta.$$
   \end{enumerate}    
    For each $x\in \overline{U}$ and small $\rho>0$,	let  $F(x)(\rho):=\{v\in F(x): \|v\|<\rho\}$.
   
   \begin{proposition} \label{Prop: Domin}
   	There exists $\hat{\varepsilon}_1>0$, such that for every $x\in \Lambda$, $n>0$ and every $C^1$ map $\phi: {\rm Dom}(\phi)\rightarrow E(x)$ with 
   	\begin{enumerate}
   		\item[$ \bullet$]  $\phi(0)=0$, ${\rm Dom}(\phi)$ is an open subset of $F(x)(\hat{\varepsilon}_1)$ and $\Lip(\phi)\leq \frac{1}{3}$,
   	\end{enumerate}   
   	there is a family of $C^1$ maps $\left \{\varphi_i:  {\rm Dom}(\varphi_i)\rightarrow E(f^i(x))\right\}_{i=0}^{n-1}$ so that the following properties hold:
   	\begin{enumerate}
   		\item[(i)]  $\varphi_i(0)=0$, ${\rm Dom}(\varphi_i)$ is an open subset of $F(f^i(x))(\hat{\varepsilon}_1)$ and $\Lip(\varphi_i)\leq \frac{1}{3}$;
   		\item[(ii)]  $F_x^i(\graph(\phi)) \supset \graph(\varphi_i)\supset F_x^i\left( \bigcap_{j=0}^{i}F_{f^j(x)}^{-j} (R_{f^j(x)}(\hat{\varepsilon}_1))\cap \graph(\phi)\right),~\forall i=0,\cdots,n-1$.
   	\end{enumerate}   
   \end{proposition}
   Before we prove the above proposition, we first consider the case of a one-step iteration.
   \begin{lemma}\label{Lem:Domin}
   	There exists $\hat{\varepsilon}_1>0$ such that 
    for every $x\in \Lambda$ and every $C^1$ map $\phi: {\rm Dom}(\phi)\rightarrow E(x)$ with 
     $\phi(0)=0$, ${\rm Dom}(\phi)$ is an open subset of $F(x)(\hat{\varepsilon}_1)$ and $\Lip(\phi)\leq \frac{1}{3}$, there is a bounded open set $V\subset F(f(x))$ and a $C^1$ map $\Phi:V\rightarrow E(f(x))$ for which 
     $\graph(\Phi)=F_x(\graph(\phi))$ and $\Lip(\Phi)\leq \frac{1}{3}$.
   \end{lemma}
   	\begin{proof}
   	Recall that the constant $C_{\rm ang}$ is chosen in \eqref{eq:Projec}, choose a number $\delta>0$ small enough such that 
   	$$0<\delta<\dfrac{1-e^{-\lambda}}{20C_{\rm ang}}\min_{x\in M}m(D_xf).$$
   	Then, there is $\hat{\varepsilon}_1>0$ such that $F_x: R_x(\hat{\varepsilon}_1)\to M$ satisfies
   	$$ \Lip(F_x-D_xf)\leq \delta ~ \text{and} ~ \sup_{v\in R_x(\hat{\varepsilon}_1)}\|D_v(F_x)-D_0(F_x)\|\leq \delta.$$
   	Let $h=F_x-D_xf$, then we have $\Lip(h)\leq \delta$. For every $u\in {\rm Dom}(\phi)$, one has that 
   	\begin{align*}
   		&\pi_{f(x)}^F (F_x(u+\phi(u)))=D_xf(u)+\pi_{f(x)}^{F} \circ h(u+\phi(u)) \\
   		&\pi_{f(x)}^E (F_x(u+\phi(u)))=D_xf(\phi(u))+\pi_{f(x)}^{E} \circ h(u+\phi(u)).
   	\end{align*}

   	 Define a map $T: {\rm Dom}(\phi)\rightarrow F(f(x))$ as $u\mapsto D_xf(u)+\pi_{f(x)}^{F} \circ h(u+\phi(u))$, one can show  that $T$ is injective. In fact, 
     for every $u_1,u_2\in {\rm Dom}(\phi)$ and $u_1\neq u_2$, we have that 
    \begin{align*}
    \|Tu_1-Tu_2\|  = &\|D_xf(u_1)-D_xf(u_2)+\pi_{f(x)}^{F} \circ h(u_1+\phi(u_1))-\pi_{f(x)}^{F} \circ h(u_2+\phi(u_2))\|\\
    	\geq &\left(m(D_xf|_{F(x)})-\|\pi_{f(x)}^F\|\cdot{\rm Lip}(h)\cdot(1+ {\rm Lip}(\phi)) \right)\|u_1-u_2\|\\
    	>&\frac{9}{10} \min_{x\in M}m(D_xf) \|u_1-u_2\|>0.
    \end{align*}
    Let $V=T({\rm Dom}(\phi))$, then one has that $V$ is a bounded open subset of $F(f(x))$ and $T:{\rm Dom}(\phi)\rightarrow V$ is a diffeomorphism. Define
    $\Phi: V \rightarrow E(f(x))$ as 
    $$\Phi(v)=D_xf(\phi(T^{-1}v))+\pi_{f(x)}^{E} \circ h(T^{-1}v+\phi(T^{-1}v)).$$
    Clearly, $\Phi: V\rightarrow E(f(x))$ is a $C^1$ map that satisfies $\graph(\Phi)=F_x(\graph(\phi))$. 
    
    To complete the proof of the lemma, we show that $\Lip(\Phi)\leq \frac{1}{3}$. For every $v_1,v_2\in V$, put $u_i=T^{-1}v_i,~i=1,2$. Then, one has
    \begin{align*}
    	\frac{\|\Phi(v_1)-\Phi(v_2)\|}{\|v_1-v_2\|}=\dfrac{\|D_xf(\phi(u_1))-D_xf(\phi(u_2))+\pi_{f(x)}^{E} \circ h(u_1+\phi(u_1))-\pi_{f(x)}^{E} \circ h(u_2+\phi(u_2))\|}{\|D_xf(u_1)-D_xf(u_2)+\pi_{f(x)}^{F} \circ h(u_1+\phi(u_1))-\pi_{f(x)}^{F} \circ h(u_2+\phi(u_2))\|}.
    \end{align*}
    Note that 
    \begin{align*}
    	&\|D_xf(u_1)-D_xf(u_2)+\pi_{f(x)}^{F} \circ h(u_1+\phi(u_1))-\pi_{f(x)}^{F} \circ h(u_2+\phi(u_2))\|\\
    	\geq &\left(m(D_xf|_{F(x)})-2C_{\rm ang}\delta \right)\|u_1-u_2\|
    \end{align*}
    and 
    \begin{align*}
    &\|D_xf(\phi(u_1))-D_xf(\phi(u_2))+\pi_{f(x)}^{E} \circ h(u_1+\phi(u_1))-\pi_{f(x)}^{E} \circ h(u_2+\phi(u_2))\|\\
    	\leq &\left(\frac{1}{3}\|D_xf|_{E(x)}\|+2C_{\rm ang}\delta \right)\|u_1-u_2\|\\
    	\leq  &\left(\frac{e^{-\lambda}}{3}m(D_xf|_{F(x)})+2C_{\rm ang}\delta \right)\|u_1-u_2\|
    \end{align*}
    where the last inequality using the dominated splitting property. 
    Consequently, on can show  that $\Lip(\Phi)\leq \frac{1}{3}$ by the choice of $\delta$. This completes the proof of this Lemma.
   \end{proof}
   
   Now, we prove Proposition \ref{Prop: Domin}.
   \begin{proof}[Proof of Proposition \ref{Prop: Domin}]
   Assume that $x\in \Lambda$, $n>0$ and $\phi:  {\rm Dom}(\phi)\rightarrow E(x)$ is a $C^1$ map  satisfying that 
   	\begin{enumerate}
   		\item[$ \bullet$]  $\phi(0)=0$, ${\rm Dom}(\phi)$ is an open subset of $F(x)(\hat{\varepsilon}_1)$ and $\Lip(\phi)\leq \frac{1}{3}$.
   	\end{enumerate}   
   	We first let $\varphi_0=\phi$ and ${\rm Dom}(\varphi_0)={\rm Dom}(\phi)$. 
   	By Lemma \ref{Lem:Domin}, there exist an open set $V(\varphi_0)\subset F(f(x))$ and a $C^1$ map $\Phi(\varphi_0):V(\varphi_0)\rightarrow E(f(x))$ such that   $\graph(\Phi(\varphi_0))=F_x(\graph(\varphi_0))$ and $\Lip(\Phi(\varphi_0))\leq \frac{1}{3}$.
   	
   Let 
   	$\varphi_1:=\Phi(\varphi_0)|_{{\rm Dom}(\varphi_1)}:{\rm Dom}(\varphi_1)\rightarrow E(f(x))$, where ${\rm Dom}(\varphi_1)=V(\varphi_0)\cap F(f(x))(\hat{\varepsilon}_1)$  is an open subset of $F(f(x))(\hat{\varepsilon}_1)$. Clearly, we have that $\Lip(\varphi_1)\leq \frac{1}{3}$ and
   	$$F_x(\graph(\varphi_0))\supset \graph(\varphi_1)\supset F_x(F_{f(x)}^{-1} (R_{f(x)}(\hat{\varepsilon}_1)) \cap  \graph(\varphi_0) ).$$

   Repeat the above process, for every $1\leq i <n$ there exist an open subset ${\rm Dom}(\varphi_i) \subset F(f^i(x))(\hat{\varepsilon}_1)$ and a $C^1$ map $\varphi_i:{\rm Dom}(\varphi_i)\rightarrow E(f^i(x))$ satisfying that $\Lip(\varphi_i)\leq \frac{1}{3}$ and
   $$F_{f^{i-1}(x)}(\graph(\varphi_{i-1}))\supset \graph(\varphi_i)\supset F_{f^{i-1}(x)}(F_{f^i(x)}^{-1} (R_{f^i(x)}(\hat{\varepsilon}_1)) \cap  \graph(\varphi_{i-1}) ).$$
   Therefore, we have that 
    $$F_x^i(\graph(\phi)) \supset \graph(\varphi_i)\supset F_x^i\left( \bigcap_{j=0}^{i}F_{f^j(x)}^{-j} (R_{f^j(x)}(\hat{\varepsilon}_1))\cap \graph(\phi)\right),~\forall i=1,\cdots,n-1.$$
    This completes the proof of Proposition \ref{Prop: Domin}.
   \end{proof}
   
   The following proposition depends on the continuity of the splitting $T_{\overline{U}}M=E\bigoplus F$.
   \begin{proposition}\label{Prop:Var}
   	There exist $\varepsilon_2>0$ and $\hat{\delta}_1>0$ such that for every $x\in \Lambda$, every $y \in B(x,\varepsilon_2)$, every $0<\hat{\delta}<\hat{\delta}_1$ and every $C^1$ map 
   	$\psi_y: F(y)(2\hat{\delta})\rightarrow E(y)$ with $\Lip(\psi_y)\leq \frac{1}{4}$ and $\psi_y(0)=0$, there exists a $C^1$ map 
   	$$\psi_x^y:~\{\pi_x^F\circ\exp_x^{-1}(y)+u:u\in F(x)(\hat \delta) \} \rightarrow E(x)$$
   	such that $\mathrm{Lip}(\psi_x^y)\leq \frac{1}{3}$ and $\exp_x({\rm graph}(\psi_x^y))\subset\exp_y({\rm graph(\psi_y)})$.
   \end{proposition}
   \begin{proof}
   	Since $U$ is an open neighborhood of $\Lambda$, there exists $\varepsilon_{0}>0$  such that $B(\Lambda,\varepsilon_{0})\subset U$. 
   	Then, for every $x\in\Lambda$, the splitting $T_yM=E(y)\bigoplus F(y)$ is well-defined  for each $y \in B(x,\varepsilon_0)$.
   	Choose $\hat{\delta}_0>0$ such that $\exp_x^{-1}\circ\exp_y: R_y(2\hat{\delta}_0)\rightarrow B_x(\rho(M))$ is well-defined for every $x\in \Lambda$ and every $y \in B(x,\varepsilon_2)$.
   	
   	We denote $A_{x,y}=D_{0}(\exp_x^{-1} \circ \exp_y)$ and $h_{x,y}=\exp_x^{-1} \circ \exp_y-D_{0}(\exp_x^{-1} \circ \exp_y)$. 
   	Since the splitting $T_{\overline{U}}M=E\bigoplus F$ is continuous, for every $\eta>0$, there exist $\varepsilon>0$ and $\hat \delta>0$ such that for any $x\in \Lambda$ and any $y \in B(x,\varepsilon)$, the map $\exp_x^{-1}\circ\exp_y: R_y(2\hat{\delta})\rightarrow B_x(\rho(M))$ restricted on $R_y(2\hat{\delta})$ satisfies that
   	\begin{equation}\label{eq:Exp}
   		\begin{aligned}
   				&m(\pi_{x}^{F}|_{A_{x,y}F(y)})\geq 1-\eta,~\|\pi_{x}^{F}|_{A_{x,y}E(y)}\|\leq \eta,~{\rm Lip}(h_{x,y})\leq \eta, \\
   				&\text{and}~\|\pi_{x}^{E}|_{A_{x,y}E(y)} \|\leq 1+\eta,~\|\pi_{x}^{E}|_{A_{x,y}F(y)}\|\leq \eta.
   		\end{aligned} 		
   	\end{equation}
   	
   	We now fix  $0<\eta<100^{-1}C_{\rm ang}^{-1}$, and choose $\varepsilon_2$ and $\hat{\delta}_1$ satisfying \eqref{eq:Exp}.
   	For $x\in \Lambda$, $y \in B(x,\varepsilon_2)$ and $\hat{\delta}<\hat{\delta}_1$, let
   	$$D_x^y:=\{\exp_x^{-1}\circ\exp_y(v+\psi_y(v)):v \in F(y)(2\hat \delta)\}\subset T_xM.$$
   	We first prove that $\pi_x^F$ is invertible when restricted to $D_x^y$. 
   	For every two distinct points $u_1, u_2\in D_x^y$, take $v_1, v_2\in F(y)(2\hat \delta)$ such that $v_1\neq v_2$ and $u_i=\exp_x^{-1}\circ\exp_y(v_i+\psi_y(v_i))\in D_x^y,~i=1,2$. 
   	By the choice of $\eta$, we have that
    \begin{equation}\label{lower-bound}
   	\begin{aligned}
   		&\|\pi^F_x(u_1-u_2)\|\\
        &\geq \|\pi_x^F A_{x,y}(v_1-v_2)\|-\|\pi_x^F A_{x,y}(\psi_y (v_1)-\psi_y (v_2))\|-\|\pi_x^F\| \cdot {\rm Lip}(h_{x,y}) \cdot 2\|v_1-v_2\| \\
   		                                      &\geq \left(m(\pi_{x}^{F}|_{A_{x,y}F(y)})- \frac{1}{4}\|\pi_{x}^{F}|_{A_{x,y}E(y)}\|-2\|\pi_x^F\| \cdot {\rm Lip}(h_{x,y})\right)\cdot \|v_1-v_2\|\\
   		                                      &\geq \left(1-4C_{\rm ang}\eta \right)\cdot \|v_1-v_2\| >\frac{4}{5} \|v_1-v_2\|>0.
   	\end{aligned}
    \end{equation}
   	This implies that $\pi_x^F$ is injective on $D_x^y$. 
   	Moreover, we have $\pi_x^F(D_x^y)\supset \{\pi_x^F\circ\exp_x^{-1}(y)+u:u\in F(x)(\hat \delta) \}.$
   	We now define $\psi_x^y$ on $\{\pi_x^F\circ\exp_x^{-1}(y)+v:v\in F(x)(\hat \delta) \}$ by $\psi_x^y(v)=\pi_x^E \circ (\pi_x^F|_{D_x^y})^{-1}(v)$.
   	Then, we have $\exp_x({\rm graph}(\psi_x^y))\subset\exp_y({\rm graph(\psi_y)})$.
   	
   	To complete the proof, we show that $\Lip(\psi_x^y)\leq \frac{1}{3}$. 
   	For every two distinct points $u_1, u_2 \in D_x^y$, take $v_1, v_2\in F(y)(2\hat \delta)$ such that $v_1\neq v_2$ and $u_i=\exp_x^{-1}\circ\exp_y(v_i+\psi_y(v_i))\in D_x^y,~i=1,2$. 
   	Note that 
   	\begin{align*}
   	\|\pi^E_x(u_1-u_2)\|&\leq \|\pi_x^E A_{x,y}(v_1-v_2)\|+\|\pi_x^E A_{x,y}(\psi_y (v_1)-\psi_y (v_2))\|+\|\pi_x^E\| \cdot {\rm Lip}(h_{x,y}) \cdot 2\|v_1-v_2\| \\
   	                                      &\leq \frac{1+15C_{\rm ang}\eta}{4}\|v_1-v_2\| 
   	\end{align*}   	
   	As in the proof of \eqref{lower-bound}, we have that  
   	  $$	\| \pi^F_x(u_1-u_2) \| \geq (1-4C_{\rm ang}\eta) \cdot \|v_1-v_2\|.$$
   	Consequently, one can show that  $\Lip(\psi_x^y)\leq \frac{1+15C_{\rm ang}\eta}{4 (1-4C_{\rm ang}\eta)}\leq  \frac{1}{3}$ by the choice of $\eta$. 
   	This completes the proof of the proposition. 
   	 \end{proof}
   	 
   	 The next result is trivial. Given $\hat{\delta}>0$, for each $x\in \Lambda$ and every $u\in F(x)$, let  $F(x)(u,\hat{\delta}):=\{v \in F(x): \|v-u\|< \hat{\delta} \}$.
   	 \begin{proposition}\label{Prop:Cover}
   	 	There exists $C_{\dim}>0$, such that for every $x\in \Lambda$ and every $0<\hat{\delta}<\hat{\varepsilon}$, there exists a family of subsets $ \{F(x)(u_i,\hat{\delta})\}_{i=1}^{M}$ satisfying that
   	 \begin{enumerate}
   	 	\item [(a)]  $\displaystyle{F(x)(\hat \varepsilon)\subset \bigcup_{i=1}^{M} F(x)(u_i,\hat{\delta})}$ and $u_i \in F(x)(\hat \varepsilon)$;
   	 	\item [(b)]  $\displaystyle{M\leq C_{\dim} \cdot \lceil \dfrac{\hat{\varepsilon}}{\hat{\delta}}\rceil^{\dim F}}$, where $\lceil \beta \rceil$ denotes the smallest integer that is larger than $\beta$.
   	 \end{enumerate}
   	 \end{proposition}
   	 
   	 \section{Proof of main Theorems}\label{Sec:PofBC}	 
   	 We choose small positive constants $\varepsilon_{f}>0$ and $\hat{\varepsilon}_f>0$ such that 
   	 \begin{equation}\label{eq:epsilonf}
   	 	\varepsilon_{f}<\varepsilon_2, \quad \hat{\varepsilon}_f<\min\{\hat{\varepsilon}_1,\hat{\delta}_1\} \quad \text{and} \quad \exp_x^{-1}B(x,2\varepsilon_f)\subset R_x(\hat{\varepsilon}_f),
   	\end{equation}
   	 where $\hat{\varepsilon}_1$ is chosen as in Proposition \ref{Prop: Domin}, $\varepsilon_2$ and $\hat{\delta}_1$ are chosen as in Proposition \ref{Prop:Var}.
   	 
   	Given an $f$-invariant ergodic measure $\nu$, by Oseledec's theorem there exists a positive integer $t(\nu)$ which is smaller than the dimension of the manifold $M$ such that 
   	 the Oseledec splitting $E^1 \bigoplus\cdots\bigoplus E^{t(\nu)}$ and Lyapunov exponents $\lambda_{1}(\nu)>\cdots>\lambda_{t(\nu)}(\nu)$ are well-defined $\nu$-almost everywhere.
   	 
   	 Assume that $\nu$ is supported on $\Lambda$ with $\lambda_F^{\min}(\nu)>0$.
   	 By the uniqueness of the Oseledec splitting, there exists $i(\nu)$ with $\lambda^{i(\nu)}(\nu)>0$, $E^{u,i(\nu)}=F$ and $E^{s,i(\nu)}=E$.
   	 Then, we denote $h_{\nu}^{F}(f):=h_{\nu}^{i(\nu)}(f)$ (recall that $E^{u/s,i(\nu)}$ is defined in \eqref{eq:SB} and $h_{\nu}^{i(\nu)}(f)$ is defined in \eqref{eq:LY}). 
     
   	 \begin{theorem}\label{Prop:Key} Let $f:M\rightarrow M$ be a $C^r\,(r>1)$ diffeomorphism  on a compact Riemannian  manifold, and let $\Lambda\subset M$ be an $f$-invariant compact subset that admits a dominated splitting $T_{\Lambda}M=E \bigoplus F$. 
   	 	Assume that $\nu$ is an $f$-invariant ergodic measure supported on $\Lambda$ with $\lambda_F^{\min}(\nu)>0$, and $\cQ$ is a finite partition with $\nu(\partial \cQ)=0$ and $\diam(\cQ)<\varepsilon_{f}$. 
        Then we have  that 
   	 	$$h_{\nu}^{F}(f)\leq h_{\nu}(f,\cQ).$$
   	 \end{theorem}
     
   	   If $\lambda_E^{\max}(\nu)\leq 0$, then $h_{\nu}^{F}(f)=h_{\nu}(f)$. In consequence, 
   	   Theorem \ref{Thm:A} follows directly from Theorem \ref{Prop:Key}.

   	   We now prove Theorems \ref{Thm:B} and  \ref{Thm:C} by assuming that Theorem \ref{Prop:Key} holds, and postpone the proof of Theorem \ref{Prop:Key} to the next section. 
       
   	   \begin{proof}[Proof of Theorem \ref{Thm:B}]
   	   Let $\{\nu_k\}_{k\ge 1}$ be a sequence of invariant measures supported on $\Lambda$. For each 
       $k\ge 1$,  assume that  $\lambda_E^{\max}(x)\leq 0$, $\lambda_F^{\min}(x)>0$ for $\nu_k$-almost every $x$ . 
   	   	Assume further that $\nu_k \rightarrow \mu$ as $k$ approaching to infinity.
   	   	
   	   	For each $k\ge 1$,
   	   	let $\nu_k=\int \nu_{k,\xi} d \hat{m}_{k}(\nu_{k,\xi})$ denote the ergodic decomposition of $\nu_k$. Let 
   	   	$$X^1_{k}\bigsqcup X^2_{k} \bigsqcup \cdots \bigsqcup X^{N_k}_k$$
   	   	be a partition of the ergodic measures such that 
   	   	\begin{equation}
   	   		\forall 1\leq i\leq N_{k},~\forall \nu,\nu'\in X_k^{i},~|h_{\nu}(f)-h_{\nu'}(f)|\leq \frac{1}{k},~{\rm dist^*}(\nu,\nu')\leq \frac{1}{k}
   	   	\end{equation}
        where ${\rm dist^*}$ is the distance on the space of all Borel probability measures on $\Lambda$.
   	   	For each $1\leq i\leq N_{k}$, if $\hat{m}_{k}(X_k^{i})>0$, we choose $\nu_{k}^{i}\in X_k^{i}$ to be an ergodic component of $\nu_{k}$ such that $\lambda_E^{\max}(\nu_{k}^{i})\leq 0$ and $\lambda_F^{\min}(\nu_{k}^{i})> 0$.  
   	  Let  $\alpha_k^i=\hat{m}_{k}(X_k^{i})$ and
   	   	$\nu'_k=\sum_{i=1}^{N_k}  \alpha_k^i \cdot \nu_{k}^{i}$,  then we have that
   	   	$$|h_{\nu_k}(f)-h_{\nu'_k}(f)|\leq \frac{1}{k},~{\rm dist^*}(\nu_k,\nu'_k)\leq \frac{1}{k}.$$
   	   	
   	   Let $\cT:=\{\nu_{k}^{i}:1\leq i\leq N_k,k\ge 1\}$, and let $\cQ$ be a finite partition such that $\mu(\partial \cQ)=0$, $\nu(\partial \cQ)=0$ for all $\nu\in \cT$, and $\diam(\cQ)<\varepsilon_{f}$. 
   	   	
   	   	By Theorem \ref{Thm:A}, for every $k\ge 1$ we have that
   	  \begin{align*}
   	  	h_{\nu_k}(f)\leq h_{\nu'_k}(f)+\frac{1}{k}=\sum_{i=1}^{N_k}  \alpha_k^i \cdot h_{\nu_{k}^{i}}(f)+\frac{1}{k}=\sum_{i=1}^{N_k}  \alpha_k^i \cdot h_{\nu_{k}^{i}}(f,\cQ)+\frac{1}{k}.
   	  \end{align*}
   	  For every $m>0$, by the definition of the metric entropy  with respect to a partition we have that
   	  $$\sum_{i=1}^{N_k}  \alpha_k^i \cdot h_{\nu_{k}^{i}}(f,\cQ)= h_{\nu'_k}(f,\cQ)\leq  \frac{1}{m} H_{\nu'_{k}}(\cQ^m).$$
   	  Since $\nu'_k\rightarrow \mu$ as $k\rightarrow \infty$, we have that $H_{\nu'_k}(\cQ^m)\rightarrow H_{\mu}(\cQ^m)$.
      Therefore, one can conclude that 
   	  $$ \limsup_{k\to\infty} h_{\nu_k}(f)\leq \frac{1}{m} H_{\mu}(\cQ^m).$$
   	  Letting $m\rightarrow \infty$, we have that 
   	   $$ \limsup_{k\to\infty} h_{\nu_k}(f)\leq h_{\mu}(f,\cQ)\leq h_{\mu}(f).$$
   	   This completes the proof of Theorem \ref{Thm:B}.
   	   \end{proof}
   	   
   	   \begin{proof}[Proof of Theorem \ref{Thm:C}]
   	   Let $\mu$ be an $f$-invariant measure supported on $\Lambda$ with
   	   $\lambda_E^{\max}(x)\leq 0$ and $\lambda_F^{\min}(x)>0$ for $\mu$-almost every $x$.   	     
   	   For every sequence of $f$-invariant measures $\{\nu_k\}_{k>0}$ with $\nu_k\rightarrow \mu$ and 
   	   $h_{\nu_k}(f)\rightarrow h$ as $k$ approaching to infinity, we show that $h\leq h_{\mu}(f)$ which implies the desired result.   

       \begin{claim}
       For every $\chi>0$, we have that 
   	   $$\nu_{k}\Big(\big\{x:\lambda_{E}^{\max}(x)<\chi\big\}\Big)\rightarrow 1,~\text{and}~\ \nu_{k}\Big(\big\{x:\lambda_{F}^{\min}(x)>0\big\}\Big)\rightarrow 1$$ 
       as $k$ approaching to infinity. 
       \end{claim}
   	   \begin{proof}[Proof of the claim]
        We provide a proof of the statement $\nu_{k}\Big(\big\{x:\lambda_{E}^{\max}(x)<\chi\big\}\Big)\rightarrow 1\,(k\to\infty)$, the other statement can be proven in a similar fashion.
       
        We will prove this result by contradiction.         
        Without loss of generality,  assume that there exists $\varepsilon_0>0$ such that 
        $$\nu_{k}\Big(\big\{x:\lambda_{E}^{\max}(x)< \chi\big\}\Big)<1-\varepsilon_0,~\forall k>0.$$
        Let $A:=\big\{x:\lambda_{E}^{\max}(x)< \chi\big\}$, and let $\alpha_k:=1-\nu_{k}(A)>\varepsilon_0$,  
         $\displaystyle{\frac{\nu_k(\cdot\cap A)}{\nu_k(A)}:=\nu_k^{+}(\cdot)}$ and $\displaystyle{\frac{\nu_k(\cdot\cap (M-A))}{\nu_k(M-A)}:=\nu_k^{-}(\cdot)}$ for each $k>0$,  then
        $$\nu_k(\cdot)=(1-\alpha_{k})\nu_k^{+}(\cdot)+\alpha_k \nu_k^{-}(\cdot).$$
        Since $A$ is an $f$-invariant subset, we have that $\nu_k^{\pm}$ are $f$-invariant measures. By passing to a subsequence, we assume that 
        $$\lim_{k\rightarrow \infty}\alpha_k=\alpha \geq \varepsilon_0,~\lim_{k\rightarrow \infty}\nu_k^{+}=\mu^{+},~\lim_{k\rightarrow \infty}\nu_k^{-}=\mu^{-}.$$
        Consequently, one has that $\mu^{+},\mu^-$ are $f$-invariant measures and $\mu=(1-\alpha) \mu^{+}+ \alpha \mu^{-}(\alpha\geq \varepsilon_0)$.

        For every $f$-invariant measure $\nu$, it follows from the sub-additive ergodic theorem  that
        $$\lambda_{E}^{\max}(\nu):=\int \lambda_{E}^{\max}(x) d\nu=\lim_{n\rightarrow \infty} \frac{1}{n} \int \log \| D_xf^n|_{E(x)}\| d\nu=\inf_{n>0} \frac{1}{n} \int \log \| D_xf^n|_{E(x)}\| d\nu.$$
        Then, the function $\nu \mapsto \lambda_{E}^{\max}(\nu)$ is upper semi-continuous on the space of all $f$-invariant measures.
        Since $\nu_{k}^{-}$ is supported on $\big\{x:\lambda_{E}^{\max}(x)\ge \chi\big\}$ for every $k>0$, it follows that
        $$\int \lambda_{E}^{\max}(x) d\mu^{-}=\lambda_{E}^{\max}(\mu^{-})\geq \limsup_{k\rightarrow \infty} \lambda_{E}^{\max}(\nu_{k}^{-})\geq \chi.$$
        Consequently, one can show that  $\mu^{-}\Big(\big\{x:\lambda_{E}^{\max}(x)\geq \chi\big\}\Big)>0$. It follows immediately that 
        $$\mu\Big(\big\{x:\lambda_{E}^{\max}(x)\geq \chi\big\}\Big)\geq \alpha \mu^{-}\Big(\big\{x:\lambda_{E}^{\max}(x)\geq \chi\big\}\Big)>0,$$
        which contradicts with the assumption that $\lambda_{E}^{\max}(x)\leq 0$ for $\mu$-almost every $x$.
       \end{proof}
   	   
   	 Given a small number $\alpha>0$, choose $\chi>0$ such that $\chi \cdot \dim E<\alpha$.    
       Let $\cS=\big\{x:\lambda^{\max}_{E}(x)\leq \chi,~ \lambda^{\min}_{F}(x)> 0\big\}$, by the Claim we have that
       $$\varepsilon_k:=1-\nu_k \left( \cS\right) \rightarrow 0~(k\rightarrow \infty).$$
   	   For each $k>0$, let $\displaystyle{\nu_{k,1}(\cdot)=\frac{\nu_k(\cdot\cap \cS)}{\nu_k(\cS)}}$ and $\displaystyle{\nu_{k,2}(\cdot)=\frac{\nu_k(\cdot\cap(M-\cS))}{\nu_k(M-\cS)}}$, then one has 
   	   $\nu_k=(1-\varepsilon_{k})\nu_{k,1}+\varepsilon_{k}\nu_{k,2}$ and
   	   \begin{align*}
   	    h:=\lim_{k\rightarrow \infty}h_{\nu_k}(f)=\lim_{k\rightarrow \infty }\Big( (1-\varepsilon_{k})h_{\nu_{k,1}}(f)+\varepsilon_{k}h_{\nu_{k,2}}(f)\Big )=\lim_{k\rightarrow \infty}h_{\nu_{k,1}}(f).
   	   \end{align*}


   	  Since $\nu_k \rightarrow \mu$ and $\varepsilon_k\rightarrow 0$, we have that $\nu_{k,1}\rightarrow \mu$ as $k\rightarrow \infty$. For the sequence of invariant measures $\{\nu_{k,1}\}$, 
      using similar arguments as  in the proof of Theorem \ref{Thm:B}, for each $k$, one can show that there exists an invariant measure $\displaystyle{\nu'_{k,1}=\sum_{j=1}^{N_k} \alpha_k^j \nu^{j}_{k,1}}$ such that  
      \begin{itemize}
           \item $\nu^{j}_{k,1}$ is an ergodic measure with $\lambda_{F}^{\min}(\nu^{j}_{k,1})>0$ and $\lambda_{E}^{\max}(\nu^{j}_{k,1})\leq \chi$ for every $1\leq j\leq N_k$;
           \item $|h_{\nu'_{k,1}}(f)-h_{\nu_{k,1}}(f)|\leq \frac{1}{k}$ and ${\rm dist^*}(\nu'_{k,1},\nu_{k,1})\leq \frac{1}{k}$.
       \end{itemize}   	   
   	   Moreover, there exists a finite partition $\cQ$ such that for every $k>0$
   	   $$h_{\nu'_{k,1}}(f)=\sum_{j=1}^{N_k} \alpha_k^j h_{\nu_{k,1}^{j}}(f)\leq \sum_{j=1}^{N_k} \alpha_k^j h_{\nu_{k,1}^{j}}^{F}(f)+\dim E \cdot \chi\leq \sum_{j=1}^{N_k} \alpha_k^j h_{\nu_{k,1}^{j}}(f,\cQ)+\alpha,$$
       where the second inequality uses the property of partial entropy (see Section \ref{Sec:entropy}), and the last inequality follows from Theorem \ref{Prop:Key}.
       
   	   Repeating the argument from the proof of Theorem \ref{Thm:B}, we conclude that $h\leq h_{\mu}(f)+\alpha$.
   	   Since $\alpha$ is chosen arbitrarily, we complete the proof of Theorem \ref{Thm:C}.
   	   \end{proof}
   	   
   	   \section{Proof of Theorem \ref{Prop:Key}} \label{Sec:PofT4}
   	   Assume that $\nu$ is an ergodic measure supported on $\Lambda$ with $\lambda_F^{\min}(\nu)>0$.
   	   By the uniqueness of the Oseledec splitting, there exists $i(\nu)$ with $\lambda^{i(\nu)}(\nu)>0$ and $E^{u,i(\nu)}=F$.  	   
       Let $W^{F}(x)$ denote the ($i(\nu)$-strongly) unstable manifold $W^{u,i(\nu)}(x)$ (see \eqref{eq:Unstable}), and let  $d^{F}_{x}$ denote the Riemannian metric inherited on $W^{F}(x)$.

       Using the Pesin theory of non-uniform hyperbolic systems \cite{Pesin07}, we have the following version of the Pesin local unstable manifold theorem.
   	   \begin{theorem}\label{Thm:Unstable}
   	   	  There exist $\lambda_{\nu}>0$ and  $\varepsilon_{\nu}>0$ such that for $\nu$-almost every $x$, there exist $\delta_x>0$, $C_x\geq 1$ and a $C^r$ map $g_x: F(x)(\delta_x)\mapsto E(x)$ with $g_x(0)=0,~\Lip(g_x)\leq \frac{1}{4}$ and satisfying 
   	   	\begin{enumerate}
   	   		\item[(a)]  $ \delta_{f^{\pm}(x)}\geq \delta_x e^{-\varepsilon_{\nu}}$, $C_{f^{\pm}(x)}\leq e^{\varepsilon_{\nu}}C_x$;
   	   		\item[(b)]  $F_x(\graph(g_x))\supset \graph(g_{f(x)})$, where $F_x$ is defined in Section \ref{Sec:Dominated};
   	   		\item[(c)]  $\exp_x (\graph(g_x))\subset W^F(x)$ and $W^F(x)=\bigcup_{n=1}^{+\infty} f^{n} \exp_{f^{-n}(x)} (\graph(g_{f^{-n}(x)}))$;
   	   		\item[(d)] for $n>0$ and $y,z\in \exp_x (\graph(g_x))$, one has $d^F_{f^{-n}(x)}(f^{-n}(y),f^{-n}(z))\leq C_x e^{-n\lambda_{\nu}} d^F_x(y,z)$.
   	   	\end{enumerate}
   	   \end{theorem}
   	 
   	   \begin{lemma}\label{return-time}
   	   	Let $(X,f,\mu)$ be an ergodic measure-preserving system,
   	   	and let  $E$ be a measurable subset with $\mu(E)>0$. 
   	   	For every $n\in \N$, let 
   	   	$$m_n(x):=m_n^{E}(x)=\max\Big\{0\leq k<n:f^k(x)\in E\Big\},$$
   	   	  then  we have that $\displaystyle{\lim_{n\to \infty}\frac{m_n(x)}{n}=1}$ for $\mu$-almost every $x$.
   	   \end{lemma}
   	   \begin{proof} 
   	   	This result may be well-known, we give a detailed proof here for completeness. 
        We use the method of contradiction to prove this conclusion.
   	   	Assume that there exists a small $\varepsilon>0$ such that  the following set 
   	   	$$K:=K_{\varepsilon}=\left\{x: \liminf_{n \rightarrow \infty} \frac{m_n(x)}{n}<1-4\varepsilon \right\} $$
        has positive $\mu$-measure.
   	   	By the Birkhoff ergodic theorem, there exists $K'\subset K$ with $\mu(K')>0$ and $N\in \N$ such that 
   	   	$$\forall n>N,\forall x\in K',~\# \Big\{0\leq i<n:f^i(x)\in E\Big\}\in \left(n(1-\varepsilon)\mu(E),n(\frac{1-2\varepsilon}{1-3\varepsilon})\mu(E)\right).$$
   	    For each $x\in K'$,   choose $n_x \in \N$ such that $m_{n_x}(x) < n_x (1-3\varepsilon)$ and $N < n_x(1-3\varepsilon)$.
   	    Hence, 
   	    $$n_x(1-2\varepsilon)\mu(E)\geq \# \Big\{0\leq j<n_x (1-3\varepsilon) :f^j(x)\in E\Big\}=\# \Big\{0\leq j<n_x :f^j(x)\in E\Big\} \geq n_x(1-\varepsilon)\mu(E).$$
   	    This contradiction yields the desired result.  	   
   	   \end{proof}
   	   
   	    Fix a finite partition $\cQ$ with $\nu(\partial \cQ)=0$ and $\diam(\cQ)<\varepsilon_{f}$.  	    
   	    Given $\alpha>0$, we will show that $h_{\nu}^{F}(f)\leq h_{\nu}(f,\cQ)+2\alpha$.  
   	    
   	    As introduced in Section \ref{Sec:entropy}, let $\xi$ be a measurable partition subordinate to $W^F$ with respect to $\nu$, and let $\{\nu_{\xi(x)}\}$ be a family of conditional measures of $\nu$ with respect to $\xi$.
   	    
   	    We choose a compact subset $E_{\nu}\subset \Lambda$, which is always called a Pesin set, such that $\nu(E_{\nu})>\frac{1}{2}$ and there exist constants $C_{\nu}$ and $\delta_{\nu}$ for which 
   	     \begin{equation}\label{eq:Cdeltanu}
   	     	C_{\nu}\geq 1,~0< \delta_{\nu} \leq \hat{\varepsilon}_{f},~\text{and}~\delta_x\geq \delta_{\nu},~C_x\leq C_{\nu}~\text{for every}~x\in E_{\nu}.
   	     \end{equation}  	    
   	     Let
   	     $$m_n^{\nu}(x)=\max\{0\leq j<n:f^j(x)\in E_{\nu}\}.$$
   	    
   	  Using Birkhoff ergodic theorem, Proposition \ref{Prop:Two Balls} and Lemma \ref{return-time}, one can choose a compact subset $K_{\nu}$ such that 
   	    \begin{enumerate}
   	    	\item [(1)]  $\nu(K_{\nu})>\frac{1}{2}$ and $K_{\nu}\subset E_{\nu}$; 
   	    	\item [(2)]  $\displaystyle{\frac{1}{n}\sum_{k=0}^{n-1} \delta_{f^k(x)} \rightarrow \nu~(n\to\infty)}$ uniformly on $K_{\nu}$;
   	    	\item [(3)]  $\displaystyle{\frac{1}{n} m_n^{\nu}(x)\rightarrow 1~(n\to\infty)}$ uniformly on $K_{\nu}$;
   	    	\item [(4)]  there exists $\rho_{\nu}>0$ such that 
   	    		$$ h_{\nu}^F(f)\leq \liminf_{n\rightarrow +\infty} -\frac{1}{n}\log \nu_{\xi(x)}\left(K_{\nu}\cap B(x,n,\rho_{\nu})\right)+\alpha, ~\forall x\in K_{\nu}.$$
   	    \end{enumerate}    	   
   	   For $\mu$-almost every $x$, let  $W^F_{\delta}(x)=\exp_x \graph(g_x|_{F(x)(\delta)})$ for each $0< \delta\leq \delta_x$. By the choice of the set $K_{\nu}$, one can choose a point $x_{\nu}\in K_{\nu}$ such that $\nu_{\xi(x_{\nu})}(K_{\nu}\cap W^F_{\delta}(x_{\nu}))>0$ for every $\delta\in (0,\delta_{\nu})$. 
   	  Fix some $\delta_{D}\in (0,\delta_{\nu})$ such that 
   	   \begin{align}\label{eq:ChoseD}
   	   	W^F_{\delta_{D}}(x_{\nu})\subset W^F_{2\delta_{D}}(x) \subset W^F_{\delta_{\nu}}(y), ~\forall x,y\in K_{\nu}\cap W^F_{\delta_{D}}(x_{\nu}).
   	   \end{align}   	   
   	   Denote by $D_{\nu}$ the disc $W^F_{\delta_{D}}(x_{\nu})$, and let 
   	   $\omega_{\nu}(\cdot):=\nu_{\xi(x_{\nu})}(\cdot \cap K_{\nu}\cap D_{\nu})/\nu_{\xi(x_{\nu})}( K_{\nu}\cap D_{\nu})$.
   	   
   	   Let $\cP_{\nu}$ be a finite partition with $\diam(\cP_{\nu})<\rho_{\nu}$ and $\nu(\partial \cP_{\nu})=0$. It follows from  Proposition \ref{Pro:unstable-entropy-partition}  that
   	   \begin{equation}\label{eq:5.0}
   	   	h_{\nu}^{F}(f)\leq  \liminf_{n\rightarrow +\infty} \frac{1}{n}H_{\omega_{\nu}}(\cP_{\nu}^n)+\alpha.
   	   \end{equation}   	   
       By the standard entropy theory, one can show that 
       \begin{align}\label{eq:5.1}
       		\frac{1}{n}H_{\omega_{\nu}}(\cP_{\nu}^n)&\leq \frac{1}{n}H_{\omega_{\nu}}(\mathcal{Q}^n)+\frac{1}{n}H_{\omega_{\nu}}(\cP_{\nu}^n|\mathcal{Q}^n),~\forall n>0. 
       \end{align}
       \begin{proposition}\label{Prop:Tail-Unstable}
       	Let $\cQ$, $\alpha$, $\omega_{\nu}$ and $\cP_{\nu}$ be chosen as described above, then 
       	$$\limsup_{n\rightarrow \infty} \frac{1}{n}H_{\omega_{\nu}}(\cP_{\nu}^n|\mathcal{Q}^n)\leq \alpha.$$
       \end{proposition}
       We now give the proof of Theorem \ref{Prop:Key} by assuming that Proposition \ref{Prop:Tail-Unstable} holds, and postpone its proof to the next section. 
       
       \begin{proof}[Proof of Theorem \ref{Prop:Key}]
       	 For each $n>0$, let $ \omega_{\nu}^n:=\frac{1}{n}\sum_{j=0}^{n-1}\omega_{\nu}\circ f^{-j}$.   For every $0<m<n$, using Lemma \ref{Lem:11} we have that
       	\begin{align*}
       		\frac{1}{n}H_{\omega_{\nu}}(\mathcal{Q}^n)\leq \frac{1}{m}H_{\omega_{\nu}^n}(\mathcal{Q}^m)+\frac{2m}{n}\log \# \mathcal{Q}
       	\end{align*}
       	
       	We claim that $\omega_{\nu}^n\rightarrow \nu$ as $n\rightarrow \infty$. 
       	In fact, note that 
       	$$\int \phi(x) d\omega_{\nu}^n=\int \frac{1}{n}\left(\phi(x)+\phi(f(x))+\cdots + \phi(f^{n-1}(x))\right)d \omega_{\nu}$$
        for each continuous function $\phi:M\rightarrow \R$.
        Since $\omega_{\nu}$ is supported on $K_{\nu}\cap D_{\nu}\subset K_{\nu}$, one has that $\frac{1}{n} \sum_{j=0}^{n-1} \phi(f^j(x))$ converges uniformly to $\int \phi d\nu$ on $K_{\nu}$. Hence, we have that
       $$\lim_{n\rightarrow  \infty} \int \phi(x) d\omega_{\nu}^n=\int \lim_{n\rightarrow  \infty} \frac{1}{n} \sum_{j=0}^{n-1} \phi(f^j(x)) d \omega_{\nu}=\int \phi d\nu.$$
       This implies that $\omega_{\nu}^n\rightarrow \nu~(n\to \infty)$. 
       Since $\nu(\partial \cQ)=0$, we get 
       $$\lim_{n\to\infty} H_{\omega_{\nu}^n}(\mathcal{Q}^m)=H_{\nu}(\mathcal{Q}^m).$$
       Therefore, we have that 
       $$\limsup_{n\rightarrow \infty} \frac{1}{n}H_{\omega_{\nu}}(\mathcal{Q}^n)\leq \frac{1}{m}H_{\nu}(\mathcal{Q}^m).$$
       By Proposition \ref{Prop:Tail-Unstable} and together with \eqref{eq:5.0} and  \eqref{eq:5.1}, we have
       $$h_{\nu}^F(f)\leq \frac{1}{m}H_{\nu}(\mathcal{Q}^m)+2\alpha.$$
       Letting $m\rightarrow \infty$, we have $h_{\nu}^F(f)\leq h_{\nu}(f,\cQ)+2\alpha$. The arbitrariness of $\alpha$ implies the desired result  of Theorem \ref{Prop:Key}.
       \end{proof}
       
       \section{Proof of Proposition \ref{Prop:Tail-Unstable}} \label{Sec:PofP5}
       To complete the proof of the main theorems in this paper, it suffices to show Proposition \ref{Prop:Tail-Unstable} in this section.

       The quantity we need to estimate is similar to the \textit{tail entropy}, which evaluates how many  Bowen balls of small radius are required to cover a  Bowen ball of fixed radius. The difference is that all of the Bowen balls here are  restricted to the $W^F_{\delta_{D}}(x_{\nu})$, which is denoted by $D_\nu$.
       
       By the definition of conditional entropy, we have 
       \begin{equation}\label{eq:60}
       	\begin{aligned}
       		H_{\omega_\nu}(\mathcal{P}_{\nu}^n|\mathcal{Q}^n)&= \sum_{Q_n \in \mathcal{Q}^n} \omega_{\nu}(Q_n) \cdot H_{\omega_{\nu}|_{Q_n}}(\mathcal{P}_{\nu}^n) \\ 
       		&\leq \sum_{Q_n \in \mathcal{Q}^n} \omega_{\nu}(Q_n) \cdot \log \#\{P_n\in \mathcal{P}_{\nu}^n :P_n\cap Q_n \cap D_{\nu}\cap K_{\nu} \neq \emptyset \},
       	\end{aligned} 
       \end{equation}       
       where the last inequality using the fact that $\omega_\nu$ is supported on $D_{\nu}\cap K_{\nu}$, and 
       $\omega_{\nu}|_{Q_n}(\cdot)=\dfrac{\omega_{\nu}(\cdot \cap Q_n)}{\omega_{\nu}(Q_n)}$ is the normalized restriction of $\omega_{\nu}$ to $Q_n$. 
       
       Choose $\beta>0$ small enough, such that 
       \begin{equation}\label{eq:beta}
       	\beta(2+\dim F \cdot \log \sup_{x\in M} \|D_xf\|)\leq \alpha.
       \end{equation}
       Since $\nu(\partial \cP_{\nu})=0$, we choose $\eta_{\nu}>0$ such that (recall \eqref{eq:epsilonf}, \eqref{eq:Cdeltanu} and \eqref{eq:ChoseD} for $\varepsilon_f,\delta_{\nu}$ and $D_{\nu}$)
       \begin{equation}\label{eq:etanu}
       	\begin{aligned}
       		&\eta_{\nu}<\min\{\delta_{\nu},\varepsilon_f\},~\nu(B(\partial {\cal P}_{\nu},\eta_{\nu}))\leq \frac{\beta}{\log \# \mathcal{P}_{\nu}},~\nu(\partial B(\partial {\cal P}_{\nu},\eta_{\nu}))=0\\
       		&\text{and for every }~x,y\in D_{\nu}\cap K_\nu,~W^{F}_{\eta_{\nu}}(x)\subset W^F_{\delta_{\nu}}(y).
       	\end{aligned}
       \end{equation}  
       Then, by the construction  of  the set $K_{\nu}$, there exists $N_{\nu}\in \N$ such that for every $n\geq N_{\nu}$ and every $x\in K_{\nu}$, the following properties hold: 
       \begin{align*}
       	\# \Big\{0\leq i<n:f^i(x)\in B(\partial {\cal P}_{\nu},\eta_{\nu})\Big\}\leq \dfrac{2n\beta}{\log \# \mathcal{P}_{\nu}} \quad \text{and} \quad m_n^{\nu}(x)\geq n(1-\beta).       	
       \end{align*}
      Let $Q_n^{j,D}:=\Big\{x\in Q_n \cap D_{\nu}\cap K_{\nu}: m_n^{\nu}(x)=j\Big\}$, then one can show that  $\# \Big\{j: Q_n^{j,D}\neq \emptyset\Big\}\leq n\beta$ for each $n\geq N_{\nu}$.
       Hence, we have that 
       \begin{equation}\label{eq:64}
       	\#\Big\{P_n\in \mathcal{P}_{\nu}^n :P_n\cap Q_n \cap D_{\nu}\cap K_{\nu} \neq \emptyset \Big\}\leq
       	 n\beta \sup_{j\in [n(1-\beta),n)}	\#\Big\{P_n\in \mathcal{P}_{\nu}^n :P_n\cap Q_n^{j,D}  \neq \emptyset \Big\}
       \end{equation}
       
       Fix $n\geq N_{\nu}$, $Q_n^{j,D}\neq \emptyset$ and $x_{\ast}\in Q_n^{j,D}$. 
       Since $\diam(\cQ)<\varepsilon_f$, one has $Q_n^{j,D} \subset B(x_{\ast},n,\varepsilon_f)$.
       By the choice of $\varepsilon_{f}$ (see \eqref{eq:epsilonf}), we have that
       $$\exp_{x_{\ast}}^{-1} B(x_{\ast},n,2\varepsilon_f)\subset \bigcap_{i=0}^{n-1} F_{f^i(x_{\ast})}^{-i}R_{f^i(x_{\ast})}(\hat{\varepsilon}_1).$$

       By the choice of $D_{\nu}$ (see \eqref{eq:ChoseD}) and $x_{\ast} \in Q_n^{j,D} \subset D_{\nu}\cap K_{\nu}$, it follows that
       $D_{\nu}\subset W^F_{\delta_{\nu}}(x_{\ast})$. 
       Let $\phi=g_{x_{\ast}}$ and ${\rm Dom}(\phi)=F(x_{\ast})(\delta_{\nu})$, where the map $g_{x_{\ast}}$ satisfies that $W^F_{\delta_{\nu}}(x_{\ast})=\exp_{x_{\ast}}(\graph (g_{x_{\ast}}|_{F(x_{\ast})(\delta_{\nu})}))$.
       Recall that $\delta_{\nu}< \hat{\varepsilon}_1$ and $\Lip(\phi)\leq \frac{1}{4}$.       
       By Proposition \ref{Prop: Domin}, there is a family of $C^1$ maps $\left \{\varphi_i: {\rm Dom}(\varphi_i)\rightarrow E(f^i(x_{\ast}))\right\}_{i=0}^{n-1}$ satisfying 
       \begin{enumerate}
       	\item[(i)]  $\varphi_i(0)=0$, ${\rm Dom}(\varphi_i)$ is an open subset of $F(f^i(x_{\ast}))(\hat{\varepsilon}_1)$ and $\Lip(\varphi_i)\leq \frac{1}{3}$;
       	\item[(ii)]  $F_{x_{\ast}}^i(\graph(\phi)) \supset \graph(\varphi_i)\supset F_{x_{\ast}}^i\left( \bigcap_{t=0}^{i}F_{f^t(x_{\ast})}^{-t} (R_{f^t(x_{\ast})}(\hat{\varepsilon}_1))\cap \graph(\phi)\right),~\forall i=0,\cdots,n-1$.
       \end{enumerate}   
       Therefore, for every $0\leq i<n$ we have 
       \begin{equation}\label{eq:varphij}
       	f^i( Q_n^{j,D})\subset \exp_{f^i(x_{\ast})} F_{x_{\ast}}^i\left( \bigcap_{t=0}^{n-1} F_{f^t(x_{\ast})}^{-t} (R_{f^j(x_{\ast})}(\hat{\varepsilon}_1))\cap \graph(\phi)\right) \subset \exp_{f^i(x_{\ast})}\graph(\varphi_i).
       \end{equation}
   Take
       \begin{align}\label{eq:deltanj}
       \delta_{n,j}=\min\Big\{ \eta_{\nu}  (20C_{\nu})^{-1},~\eta_{\nu} 20^{-1}C_f^{j-n},~\delta_{\nu}\Big\},
       \end{align}
       where $C_f:=\sup_{x\in M}\|D_xf\|$, $\eta_{\nu}$ is chosen as in \eqref{eq:etanu} and $C_{\nu},\delta_{\nu}$ are chosen as in \eqref{eq:Cdeltanu}.

       For every $x\in Q_n^{j,D}$, 
        $W^F_{\delta_{n,j}}(f^j(x))$ is well defined since $f^j(x)\in E_{\nu}$ and $\delta_{n,j}\leq \delta_{\nu}$.
       For every $y,z\in f^{-j}W^F_{\delta_{n,j}}(f^j(x))$, if $0\leq i\leq j$, then
       \begin{align*}
         d^F_{f^{i}(x_{\ast})}(f^i(y),f^i(z))\leq C_{\nu}e^{-(j-i)\lambda_{\nu}} d^F_{f^{j}(x_{\ast})}(f^j(y),f^j(z))\leq 5C_{\nu}e^{-(j-i)\lambda_{\nu}} \delta_{n,j}\leq \frac{\eta_{\nu}}{4}.
       \end{align*}
       If $j< i< n$, then
       $$d^F_{f^{i}(x_{\ast})}(f^i(y),f^i(z))\leq C_f^{i-j}d^F_{f^{j}(x_{\ast})}(f^j(y),f^j(z))\leq 5C_f^{i-j}\delta_{n,j}\leq \frac{\eta_{\nu}}{4}.$$
       Consequently, $d^F_{f^{i}(x_{\ast})}(f^i(y),f^i(z))\leq \frac{\eta_{\nu}}{4}$ for every $0\leq i<n$. Therefore, one has
       \begin{equation}\label{eq:diam}
           \diam( f^{-j+i}W^F_{\delta_{n,j}}(f^j(x)))\leq \frac{\eta_{\nu}}{4},~\forall 0\leq i<n.
       \end{equation}
       
       \begin{lemma}\label{Lem:Covering}
       	There exists a family of points $\{x_k\}_{k=1}^{M}\subset Q_n^{j,D}$ with 
       $$M\leq C_{\dim}\cdot \lceil \dfrac{4\hat{\varepsilon}_1}{\delta_{n,j}}\rceil^{\dim F} \quad \text{and} \quad \bigcup_{k=1}^{M} W^F_{\delta_{n,j}}(f^j(x_k))\supset f^j(Q_n^{j,D}).$$
       \end{lemma}
       \begin{proof}
       	Recall that $f^j( Q_n^{j,D})\subset \exp_{f^j(x_{\ast})}\graph(\varphi_j)$ and ${\rm Dom}(\varphi_j)\subset F(f^j(x_{\ast}))(\hat \varepsilon_1)$.     By Proposition \ref{Prop:Cover}, there exists a collection $\{F(f^j(x_{\ast}))(u_k, \frac{1}{4}\delta_{n,j})\}_{k=1}^{M}$ satisfies
        $$F(f^j(x_{\ast}))(\hat \varepsilon_1)\subset \bigcup_{k=1}^{M} F(f^j(x_{\ast}))(u_k,\frac{1}{4}\delta_{n,j}),~u_k\subset F(f^j(x_{\ast}))(\hat \varepsilon_1),~M\leq C_{\dim} \cdot \lceil \dfrac{4\hat{\varepsilon}_1}{\delta_{n,j}}\rceil^{\dim F}.$$    	
       		For $x\in Q_n^{j,D}$, let $$u_x=\pi_{f^j(x_{\ast})}^F\circ\exp_{f^j(x_{\ast})}^{-1}(f^j(x))\in {\rm Dom}(\varphi_j)\subset F(f^j(x_{\ast}))(\hat \varepsilon_1).$$
            Let $\psi_{f^j(x)}:F(f^j(x))(\delta_{n,j})\rightarrow E(f^j(x))$ satisfies $W^F_{\delta_{n,j}}(f^j(x))=\exp_{f^j(x)} \graph(\psi_{f^j(x)})$.
       		Recall that  $d(f^j(x),f^j(x_{\ast}))\leq \varepsilon_2$, $\delta_{n,j}\leq \delta_{\nu}< \hat{\delta}_1$ and $\Lip(\psi_{f^j(x)})\leq \frac{1}{4}$.        		      		
       		By Proposition \ref{Prop:Var}, there exists 
       		$$\phi_{x}^{\ast}=\psi_{f^j(x_{\ast})}^{f^j(x)}: F(f^j(x_{\ast}))(u_x,\frac{1}{2}\delta_{n,j})\rightarrow E(f^j(x_{\ast}))~\text{such that}~\exp_{f^j(x_{\ast})}\graph(\phi_x^{\ast})\subset W^F_{\delta_{n,j}}(f^j(x)).$$
            Let $\cI:=\Big\{1\le k\le M: \{u_x:x\in Q^{j,D}_n\} \cap  F(f^j(x_{\ast}))(u_k,\frac{1}{4}\delta_{n,j})\neq \emptyset \Big\}$.
       		For every $k\in \cI$, we choose a point $x_k\in Q^{j,D}_n$ such that $u_{x_k}\in F(f^j(x_{\ast}))(u_k,\frac{1}{4}\delta_{n,j})$.        		
       		Consider the family of local unstable manifolds $\{W^F_{\delta_{n,j}}(f^j(x_k))\}_{k\in \cI}$.
       		We will show that $\bigcup_{k\in \cI} W^F_{\delta_{n,j}}(f^j(x_k))\supset f^j(Q_n^{j,D})$.
       		
       		Assume that there exists some $f^j(x)\in f^j(Q_n^{j,D})-\bigcup_{k\in \cI} W^F_{\delta_{n,j}}(f^j(x_k))$.
       		Take $k \in \cI$ such that 
       		$$u_x\in F(f^j(x_{\ast}))(u_k,\frac{1}{4}\delta_{n,j})\subset F(f^j(x_{\ast}))(u_{x_k},\frac{1}{2}\delta_{n,j}).$$       		
       		Recall that 
            $$f^j(x)\in \exp_{f^j(x_{\ast})}\graph(\varphi_j)
       		,~f^j(x)\notin W^F_{\delta_{n,j}}(f^j(x_k)),~ \exp_{f^j(x_{\ast})}\graph(\phi_{x_k}^{\ast})\subset W^F_{\delta_{n,j}}(f^j(x_k)).$$ Hence, we have $\phi_{x_k}^{\ast}(u_x)\neq \varphi_j(u_x)$.
       		Let $f^j(x')=\exp_{f^j(x_{\ast})}(\phi^{\ast}_{x_k}(u_x)+u_x)$. Then, one has 
            \begin{equation}\label{eq:Lem61}
            f^j(x')\notin \exp_{f^j(x_{\ast})}\graph(\varphi_j).
            \end{equation}
            
       		Note that $f^j(x') \in W^F_{\delta_{n,j}}(f^j(x_k))$. By the choice of $\delta_{n,j}$ (see \eqref{eq:deltanj}), we have $x' \in  W^F_{\eta_{\nu}}(x_k)$. 
            By the choice of $\eta_{\nu}$ (see \eqref{eq:etanu}) and $x_k\in D_{\nu}$, 
       		we have $x'\in W^F_{\delta_{\nu}}(x_{\ast})=\exp_{x_{\ast}} \graph(\phi)$.
       		
       		For every $0\leq i<n$, by \eqref{eq:diam} we have $d(f^i(x_k),f^i(x'))< \eta_{\nu}< \varepsilon_f$,  and $d(f^i(x_k),f^i(x_{\ast}))< \varepsilon_f$ since $\diam \cQ <\varepsilon_f$. 
       		Hence, $d(f^i(x_{\ast}),f^i(x'))< 2\varepsilon_f$ for any $0\leq i<n$. 
       		Therefore, by the choice of $\varphi_j$ (see \eqref{eq:varphij}), we have
       		$$\exp_{x_{\ast}}^{-1}(x')\in \bigcap_{i=0}^{n-1}F_{f^i(x_{\ast})}^{-i} (R_{f^i(x_{\ast})}(\hat{\varepsilon}_1))\cap  \graph(\phi)\subset F_{f^j(x_{\ast})}^{-j} \graph(\varphi_j).$$      		
       		Consequently, we get that $f^j(x')\in \exp_{f^j(x_{\ast})} \graph(\varphi_j)$,
       		which contradicts with \eqref{eq:Lem61}.
       		This  completes the proof of this Lemma.
       \end{proof}
             
       \begin{proof}[Proof of Proposition \ref{Prop:Tail-Unstable}]
       	We first estimate $\#\Big\{P_n\in \mathcal{P}_{\nu}^n :P_n\cap Q_n^{j,D}  \neq \emptyset\Big \}$.
       	By Lemma \ref{Lem:Covering} there exists a family of sub-discs $\{W_k\}_{k=1}^{M}$ with 
       	
       	$$M\leq C_{\dim}\cdot \lceil \dfrac{4\hat{\varepsilon}_1}{\delta_{n,j}}\rceil^{\dim F},~\bigcup_{k=1}^{M} f^j(W_k )\supset f^j(Q_n^{j,D})~\text{and}~\diam(f^iW_k)\leq \frac{\eta_{\nu}}{4},~\forall 0\leq i<n.$$       	
       	Then, we have 
       	\begin{align}\label{eq:61}
       		\#\Big\{P_n\in \mathcal{P}_{\nu}^n :P_n\cap Q_n^{j,D}  \neq \emptyset\Big \}\leq M \sup_{k} \#\Big\{P_n\in \mathcal{P}_{\nu}^n :P_n\cap W_k \cap K_{\nu}  \neq \emptyset \Big\}.
       	\end{align}
       	For every $W\in \{W_k\}_{k=1}^{M}$, we have
       	$$\#\Big\{P_n\in \mathcal{P}_{\nu}^n :P_n\cap W \cap K_{\nu} \neq \emptyset \Big\} \leq  \prod_{i=0}^{n-1} \# \Big\{P \in \mathcal{P}_{\nu}: P\cap f^i (W \cap K_{\nu}) \neq \emptyset \Big\}.$$
       	Fix some point $z\in W \cap K_{\nu}$. If 
       	$\# \{P \in \mathcal{P}_{\nu}: P\cap f^i (W \cap K_{\nu}) \neq \emptyset \}>1$,
       	then by $\diam(f^iW)\leq \frac{\eta_{\nu}}{4}$ we must have $f^i(z)\in f^i (W \cap K_{\nu})\subset B(\partial \cP_{\nu},\eta_{\nu})$.       	
       	Hence, we have
       	$$\#\{P_n\in \cP_{\nu}^n: P_n\cap W\cap K_{\nu} \neq \emptyset \} \leq (\# \cP_{\nu})^{\#\{0\leq i<n:f^i(z)\in B(\partial \cP_{\nu},\eta_{\nu} )\}}\leq (\# \mathcal P_{\nu})^{2n\beta/\log \# \mathcal{P}_{\nu}}=e^{2n\beta}.$$
       	Consequently, by \eqref{eq:64} and \eqref{eq:61} we get that
       	$$\#\Big\{P_n\in \mathcal{P}_{\nu}^n :P_n\cap Q_n \cap D_{\nu}\cap K_{\nu} \neq \emptyset \Big\}\leq n\beta \cdot C_{\dim}\cdot \lceil \dfrac{4\hat{\varepsilon}_1}{\delta_{n,j}}\rceil^{\dim F} \cdot e^{2n\beta}.$$
       	By \eqref{eq:60} and the choice of $\delta_{n,j}$ and $\beta$ (see \eqref{eq:deltanj} and \eqref{eq:beta}) we have 
       	\begin{align*}
       		\limsup_{n\to\infty} \frac{1}{n}H_{\omega_\nu}(\mathcal{P}_{\nu}^n|\mathcal{Q}^n)\leq \beta(2+\dim F\cdot \log C_f) \leq \alpha.
       	\end{align*}
       This completes the proof of Proposition \ref{Prop:Tail-Unstable}.
       \end{proof}
\section*{Acknowledgement}
 The authors would like to thank professor D. Yang for his suggestions.
\printbibliography    
\end{sloppypar}

\flushleft{\bf Chiyi Luo} \\
\small School of Mathematics and Statistics, Jiangxi Normal University, Nanchang, 330022, P.R. China\\
\textit{E-mail:} \texttt{luochiyi98@gmail.com}\\

\flushleft{\bf Wenhui Ma} \\
\small School of Mathematical Sciences,  Soochow University, Suzhou, 215006, P.R. China\\
\textit{E-mail:} \texttt{20234207032@stu.suda.edu.cn}\\

\flushleft{\bf Yun Zhao} \\
\small School of Mathematical Sciences,  Soochow University, Suzhou, 215006, P.R. China\\
\textit{E-mail:} \texttt{zhaoyun@suda.edu.cn}\\

\end{document}